\documentclass[11pt]{amsart}
\usepackage{times}
\usepackage{amsfonts,amssymb,amsmath,amsgen,amsthm,latexsym}
\usepackage{epsf,exscale,times}
\usepackage[dvips]{graphicx}
\usepackage{graphics,epsfig}
\usepackage{epsfig,rotate,color}
\usepackage{subfigure}  
\usepackage{hyperref}
\usepackage{array}
\usepackage[left=3.1cm,right=3.1cm,top=3.5cm,bottom=3.6cm]{geometry}

\theoremstyle{plain}
\newtheorem{theorem}{Theorem}[section]
\newtheorem{definition}[theorem]{Definition}

\newtheorem{lemma}[theorem]{Lemma}

\newtheorem{proposition}[theorem]{Proposition}

\theoremstyle{remark}
\newtheorem{remark}[theorem]{Remark}

\def\C{{\mathbf C}}
\def\R{{\mathbf R}}
\def\N{{\mathbf N}}

\newcommand{\SR}{\mathcal{S}(\mathbf{R})}

\def\I{\mathcal{I}}
\def\J{\mathcal{J}}

\def\F{\mathcal{F}}
\def\P{\mathbb{P}}

\def\v{\mathcal{V}}

\def\P{\mathcal{P}}

\def\({\left(}
\def\){\right)}
\def\<{\left\langle}
\def\>{\right\rangle}

\numberwithin{equation}{section}

\begin{document}

\title[FEM for the Fowler equation]{Finite Element Method for a space-fractional anti-diffusive equation}

\author[A. Bouharguane]{Afaf Bouharguane}
\address{Institut de Mathématiques de Bordeaux, Talence F-33040, 
MEMPHIS, INRIA Bordeaux Sud-Ouest  \\}
\email{afaf.bouharguane@math.u-bordeaux.fr}
\keywords{Fractional anti-diffusive operator, finite element method, Crank-Nicolson scheme, stability, error
analysis.} 

\maketitle

\date{} 

\begin{abstract}
The numerical solution of a nonlinear and space-fractional anti-diffusive equation used to model dune morphodynamics is considered. Spatial discretization is effected using a finite element method whereas the Crank-Nicolson scheme is used for temporal discretization. The fully discrete scheme is analyzed to determine stability condition  and also to obtain error estimates for the approximate solution. Numerical examples are presented to illustrate convergence results. 

\end{abstract}

\section{Introduction} 
\noindent We consider the Fowler equation \cite{Fo01}
\begin{equation}
    \begin{aligned}
        &\partial_t u(t,x) + \partial_x\left(\frac{u^2}{2} \right) (t,x)  - \partial_{xx} u (t,x) + \I[u]  (t,x)  = 0,\quad  x \in \R, t>0, 
    \end{aligned}
\label{fowlereqn}
\end{equation} 
where $\I$
is a nonlocal operator defined as follows: for any Schwartz
function $\varphi \in \SR$ and any $x \in \R$,
\begin{equation}
\I [\varphi] (x) := \int_{0}^{+\infty} |\xi|^{-\frac{1}{3}}
\varphi''(x-\xi) \, d\xi . 
\label{nonlocaltermi}
\end{equation}

The Fowler equation was introduced to model the formation and dynamics of sand structures such as dunes and ripples \cite{Fo01}. This equation is valid for a river flow over an erodible bottom $u(t,x)$ with slow variation. 
Its originality resides in the nonlocal term, wich is anti-dissipative, and can been seen as a fractional Laplacian of order $4/3$. Indeed, it has been proved in \cite{AAI10} that 
\begin{equation*}
\F( \I[\varphi])(\xi ) = -4 \pi^2 \Gamma(\frac{2}{3}) \left( \frac{1}{2} - i \mbox{sgn}(\xi) \frac{\sqrt 3}{2}  \right) |\xi|^{4/3} \F(\varphi)(\xi),
\end{equation*}
where $\Gamma$ is the gamma function and $\F$ denotes the Fourier transform. \\
Therefore, this term has a deregularizing effect on the initial data but the instabilities produced by the nonlocal term are controled by the diffusion operator $-\partial_x^2 $  which ensures the existence and the uniqueness of a smooth solution. We then always assume that there exists a sufficiently regular solution $u(t,x)$. \\  
The use of Fourier transform is a natural way to study this equation but it also can be useful to consider the following formula: \\
for all $r>0$ and all $\varphi \in \mathcal{S}(\R)$,
\begin{eqnarray}
\I[\varphi](x) &=& \I_1[\varphi](x) + \I_2[\varphi](x),  
\label{decompo}
\end{eqnarray}
with 
\begin{equation*}
 \I_1[\varphi](x) = \int_0^r |\xi|^{-1/3} \varphi''(x-\xi) \, d\xi
\end{equation*}
and
\begin{eqnarray*} 
\I_2[\varphi](x) 
&=& -\frac{1}{3} \int_{r}^{\infty} |\xi|^{-4/3} \varphi'(x-\xi) \, d\xi + \varphi'(x-r) r^{-1/3}.
\end{eqnarray*}
Several numerical approaches have been suggested in the literature to overcome the 
 equations with nonlocal operator. Droniou used a  general class of difference methods for fractional conservation laws \cite{JD10}, Zheng and Roop proposed a finite element method to solve a space-fractional advection equations \cite{Zheng10}, \cite{Roop06}. Liu proposed a numerical solution for the fractional fokkerplanck equation \cite{Liu04}. Meerschaert studied  finite difference approximations of fractional advection dispersion flow equation \cite{Meerschaert04}. Fix presented a least squares finite-element approximations of a fractional order differential equation  \cite{Fix04}.  Xu applied the discontinuous Galerkin method to  fractional convection diffusion equations with a fractional Laplacian of order $\lambda \in (1,2)$ \cite{Xu13} and, recently Guan investigated stabitlity and error estimates for  $\theta$ schemes for finite element discretization of the space-time fractional diffusion equations \cite{GG15}. \\
To solve the Fowler equation \eqref{fowlereqn} some numerical experiments have been performed using mainly finite difference method and split-step Fourier method \cite{ABRC14}.\\ 
We propose here to use the standard Galerkin method for the space approximation and a Crank-Nicolson scheme for the time discretization, which is a more simple way  to improve approximations and to model complex geometries.  \\ 
For $T>0, L>0$, we seek a function $u$ defined on $\R\times [0,T]$, 2L-periodic in the second variable and satisfying  \\
\begin{equation}
\left\{
    \begin{aligned}
        &\partial_t u(t,x) + \partial_x\left(\frac{u^2}{2} - \partial_{x} u + \J[u] \right) (t,x)  = 0,\quad  x \in \R, t \in (0,T), \\
        &u(0,x) = u_0(x),\quad  x \in \R,
    \end{aligned}
\right.
\label{fowlereqn2}
\end{equation} 
where $u_0$ is a given 2L-periodic function and 
\begin{equation}
\J [\varphi] (x) := \int_{0}^{+\infty} |\xi|^{-\frac{1}{3}}
\varphi'(x-\xi) \, d\xi .  \\
\label{nonlocalterm}
\end{equation} 
To prove the convergence of the numerical scheme we use the standard material on the finite element method for parabolic problems \cite{Thomee2006}. However, the analysis of the variational solution to the Fowler equation is more complicated than the usual parabolic equations because the fractional differential operator is not local and is anti-diffusive. \\
In this paper we analyze the discretization of \eqref{fowlereqn2} by a Crank-Nicolson method in time combined with the standard Garlerkin-finite element method in space. Our main result consists in prove the following error estimate: 
\begin{equation*}
||u(t^n,\cdot) - U^n|| \leq C (\Delta t^2 + h^k),
\end{equation*}
where $k$ is the optimal spatial rate of convergence in $L^2$, 
 $\Delta t=T/N$ is the time step, $t^n=n \Delta t,  n = 0, \cdots, N$ and $h$ is the spatial discretization. $U^0,\cdots,U^N$ are the approximations of the solutions at different times. \\
We also prove that our numerical scheme is stable if the following condition is satisfied: 
\begin{equation*}
C_1 \frac{\Delta t}{ h^2} + C_2 \frac{\Delta t}{h^{4/3}} \leq 1,
\end{equation*}
where $C_1,C_2$ are two positive constants independent of $\Delta t$ and $h$. \\

It is clear that our analysis can easily be extended to the case where the nonlocal term $\I$ is replaced with a Fourier multiplier homogeneous of degree $\lambda \in ]1,2[$ and not only $\lambda=4/3$ .  It also can replaced with the Riemann-Liouvillle integral. Indeed, for causal functions, our nonlocal term is, up to a multplicative constant, a Riemann-Liouville operator defined as follows: 
\begin{equation*}
\frac{d^{4/3} \varphi }{dx }(x) = \frac{1}{\Gamma(2/3)} \int_{0}^{+\infty} |\xi|^{-1/3} \varphi''(x-\xi) \, d\xi . \\
\end{equation*}


The rest of this paper is construct as follows. In the next section we give the preliminary knowledge regarding the fractional operator and some technical Lemmas.  We also introduce a projection operator and derive some error estimates which will play an important role in the sequel. The error estimate for the Galerkin-finite element method  to solve the problem \eqref{fowlereqn2} is studied in Section \ref{spdisc}. In section \ref{tpdisc}, we derive error estimates and prove   existence and uniqueness  of the fully discrete approximations. We also give a stability result. \\
 We finally perform some numerical experiments to confirm the theoretical results in section \ref{senum}.  \\

\subsection{Notations }
\begin{itemize}
\item We denote by $C(c_1, c_2, . . .)$ a generic positive constant, strictly positive, which depends on parameters $c_1, c_2, \cdots$
\item For $m \in \N$, let $H^m_{per}$ be the periodic Sobolev space of order $m$, consisting of the $2L-$periodic elements of $H^m_{loc}(\R).$ We denote by $||\cdot||_m$ the norm over a period in  $H^m_{per}$, by $||\cdot||$ the norm in $L^2(-L,L)$, and by $(\cdot, \cdot)$ the inner product in $L^2(-L,L)$. 
\item  We denote by $C_n(\varphi)$ the Fourier coefficient of $\varphi$ defined by: for all $n \in \mathbb{Z}$
\begin{equation*}
C_n(\varphi) = \frac{1}{2L} \int_{-L}^L \varphi(x) e^{- i  \frac{n}{L} x} \, dx 
\end{equation*}
\end{itemize}

\section{Preliminaries}

In this section, we give the variational formulation of the problem \eqref{fowlereqn2} and we introduce a projection operator. We derive some estimates wich will be useful in the next sections. \\

We shall discretize \eqref{fowlereqn2} in space by the Galerkin method. To this effect, let $-L=x_0<x_1<\cdots<x_N=L$ be a partition of $[-L,L]$ and $h:=\max_{j}(x_{j+1}-x_j)$. \\ 
 For integer $r\geq 2$, let $\mathcal{S}_h^r$ denote a space  of continuously differentiable, 2L-periodic functions of degree $r-1$ in which approximations to the solution $u(t,\cdot)$ \eqref{fowlereqn2} will be sought for $t \in [0,T]$. \\
 We assume that this family is a finite-dimensional subspaces of $H^1_{per}$ such that, for some integer $r\geq 2$ and small $h$,
\begin{equation}
\label{estimateSh}
\inf_{\chi \in S^r_h} \left\{ ||v-\chi|| + h ||\nabla(v-\chi) || \right\} \leq C h^s ||v||_s, \quad \mbox{for } 1\leq s \leq r, 
\end{equation}
where $v \in H^s_{per}$ (cf. e.g \cite{Ak94} and references therein ). \\

\noindent Note that since the pratical implementation of the scheme requires to make some truncations including the integral operator $\J$, we replace $\int_0^{+\infty}$ with $\int_0^L$ in \eqref{nonlocalterm}.  \\

\noindent A variational form of the problem is: 
\begin{equation}\label{formva}
(u_t,v) + (u u_x,v) + (u_x,v') - (\J[u],v') = 0 \quad \forall v \in H^{1}_{per}, \forall t \in (0,T).  \\
\end{equation}

\begin{proposition}[$L^2$-estimate\label{prop1}] 
Let u the solution of the variational form  \eqref{formva}. Then, 
for all $t \in [0,T]$, 
\begin{equation*}
||u(t,\cdot) ||  \leq  e^{w_0 t} ||u_0||, 
\end{equation*}
where $w_0$ is a positive constant. 

\end{proposition}

\begin{proof}
Taking $v=u(t,\cdot)$ in \eqref{formva}, we obtain by periodicity 
\begin{equation}
\frac{1}{2}\frac{d}{dt} ||u(t,\cdot)||^2 + (u_x-\J[u],u_x) = 0.
\label{L2estimate}
\end{equation}
Using the Fourier analysis, we have $C_n(u_x)  =  i \pi \frac{n}{L} C_n(u)$ and since $\J[u] = \psi \ast u_x $, with $$\psi(x) = x^{-1/3} \chi_{(0,\infty)}, $$
then $C_n(\J[u]) = C_n(\psi) C_n(u_x).$ But since
\begin{eqnarray*}
C_n(\psi) &=& \frac{1}{2L} \int_0^L x^{-1/3} e^{-i \pi \frac{n}{L}x} dx = \frac{1}{2L^{1/3}} \frac{1}{ \pi^{2/3}}  n^{-2/3} \int_0^{\pi  n } \frac{e^{-iu}}{u^{1/3}} du
\end{eqnarray*}
then, 

\begin{eqnarray*}
(u_x-\J[u],u_x) &=& \sum_{n=-\infty}^{+\infty} [ (\frac{\pi n}{L})^2 - (\frac{\pi n}{L})^{4/3} \frac{1}{2L}\int_{0}^{ \pi  n} \frac{e^{-iu}}{u^{1/3}} du ] |C_n(u)|^2 \\
&\geq&  \sum_{n=-\infty}^{+\infty} [ ( \frac{\pi n}{L} )^2 - |(\frac{\pi n}{L})^{4/3}  \frac{1}{2L}\ \int_{0}^{\pi   n} \frac{e^{-iu}}{u^{1/3}} du| ] |C_n(u)|^2 \\
\end{eqnarray*}
Since
$$\int_0^\infty \frac{\cos(u)}{u^{1/3}} du = \frac{1}{2}\Gamma(\frac{2}{3}), \mbox{ and} \int_0^{\infty} \frac{\sin(u)}{u^{1/3}} du =  \frac{\sqrt{3}}{2}\Gamma(\frac{2}{3})$$ 
it follows that 
\begin{equation*}
 | \frac{1}{2L} \int_0^{ \pi  n } \frac{e^{-iu}}{u^{1/3}} du| \leq C,
\end{equation*}
where $C $ is a positive constant. 
Therefore by Plancherel's formula,
\begin{equation*}
(u_x-\J[u],u_x) \geq  \sum_{n=-\infty}^{+\infty} [ ( \frac{\pi n}{L})^2 - (\frac{\pi n}{L})^{4/3} C )] |C_n(u)|^2 \geq -w_0 ||u(t,\cdot)||^2,
\end{equation*}
where $-w_0 = \min_{n} [(\frac{\pi n}{L})^2 - (\frac{\pi n }{L})^{4/3} C] \leq 0$. 
Finally, using \eqref{L2estimate}, we obtain
\begin{equation*}
||u(t,\cdot) || \leq  e^{w_0 t} ||u_0||. 
\end{equation*}
The proof of this proposition is now complete. \\
\end{proof}

\begin{remark} \label{rem1} Following the same lines as the proof of the Proposition \ref{prop1}, we have that:  \\
$\forall \nu > 0, \exists \alpha >0 $ such that 
\begin{equation*}
(\nu u_x-\J[u],u_x) \geq  -\alpha ||u||^2. 
\end{equation*}

\end{remark}

\begin{lemma} \label{estimateJ1}
Let $\varphi \in H^{2/3}_{per}$. Then 
\begin{equation} 
||\J[\varphi|| \leq C ||\varphi||_{2/3}.
\end{equation} 
\end{lemma}

\begin{proof}
From Fourier analysis and using computations from the Proposition \ref{prop1}, we have 
\begin{eqnarray*}
||\J[\varphi||  &=& \sum_{n} |C_n(\J)|^2 = \sum_n |C_n(\psi) C_n(\varphi)|^2 \\
&\leq& C \sum_n n^{4/3} | C_n(\varphi)|^2, \\
&=& C \sum_n \left( \frac{n^2}{1+n^2} \right)^{2/3} (1+n^2)^{2/3} | C_n(\varphi)|^2, \\
&\leq& C \sum_n (1+n^2)^{2/3} | C_n(\varphi)|^2, \\
&=& C ||\varphi||_{2/3}. 
\end{eqnarray*}
\end{proof}

\begin{lemma}[Bilinear form\label{biliforme}] Let  $u,v \in H^1_{per} $. Then, it exists  $\lambda >0$ such that the bilinear form  
$$a(u,v) = (u',v') - (\J[u],v') + \lambda (u,v)$$ 
is continuous and coercive. 
\end{lemma}
\begin{proof}
Using Lemma \ref{estimateJ1}, we can easily see that $a$ is continuous. Let us now check the coercivity. \\
Fom Remark \ref{rem1},  it exists $\alpha_0>0$ such that for all  $v \in H^1_{per} $
\begin{eqnarray*}
a(v,v) &=& \frac{1}{2}||v_x||^2  + (\frac{1}{2}v_x- \J[v],v_x) + \lambda ||v||^2, \\
&\geq& \frac{1}{2}||v_x||^2 + (\lambda-\alpha_0) ||v||^2, 
\end{eqnarray*} 
Therefore, for 
\begin{equation}
\lambda > \alpha_0,
\label{condlambda}
\end{equation} 
$a$ is coervice. \\
\end{proof}

\begin{lemma}[Projection] \label{lemproj}
We define the projection operator $\mathcal{P}: H^1_{per} \rightarrow \mathcal{S}_h^r$ by 
\begin{equation}
(v'-(\P v)',\chi') - (\J[v]-\J[\P v],\chi') + \lambda (v-\P v,\chi ) = 0, \forall \chi \in \mathcal{S}_h^r,
\label{operproj}
\end{equation}
where $\lambda $ satisfies the condition \eqref{condlambda}.  Then for all  $ 1 \leq s \leq r$ and for all $v \in H^s_{per}$, we have 
\begin{itemize}
\item[1.] $||(v-\P v)'|| \leq C h^{s-1} ||v||_s$
\item[2.] $||v-\P v|| \leq C h^{s} ||v||_s$
\end{itemize}
\end{lemma}

\begin{proof}
\noindent 1. Arguing as the proof of the Cea's Lemma and from \eqref{estimateSh}, we get for all $v \in H^s_{per}$
\begin{equation}
||v-\P v||_1 \leq C h^{s-1} ||v||_s. 
\label{cea}
\end{equation}
Indeed, using the bilinear form $a$ defined in  Lemma \ref{biliforme} we have
\begin{eqnarray*}
a(v-\P v,v-\P v) &=& a(v-\P v,  v- \chi + \chi - \P v  ) = a(v-\P v, v - \chi) \quad    \forall \chi \in \mathcal{S}_h^r,
\end{eqnarray*}
and from the coecivity and continuity properties, we get
\begin{eqnarray*}
C ||v-\P v||^2_1 &\leq&  || (v-\P v)' || \,  ||(v-\chi)'||    + ||\J[v-\P v]|| \, ||(v-\chi)'|| + \lambda ||v-Pv||\, ||v-\chi|| \\
&\leq& C || v-\P v||_1  \left(  ||(v-\chi)'||  + ||(v-\chi)'||  + \lambda ||v-\chi||  \right)
\end{eqnarray*} 
Therefore, $ ||v-\P v||_1 \leq \inf_{\chi \in S_h} ||(v-\chi)'||$, and using 
finally  the property of $\mathcal{S}_h^r$ \eqref{estimateSh}, we obtain 
$$||v-\P v||_1 \leq \C h^{s-1} ||v||_s, \quad \forall v \in H^s_{per} .$$\\

\noindent 2. To estimate $||v-\P v||$ we consider the auxiliary problem
$$a(\psi, \varphi) = (v-\P v,\varphi).$$
Then, for $\chi \in \mathcal{S}^r_h,$ we have from continuity of $a$, assumption \eqref{estimateSh} and  estimate \eqref{cea}
\begin{eqnarray*}
||v-\P v||^2 &=& a(\psi-\chi,v-Pv) \leq C  \inf_{\chi \in S^r_h} ||\psi-\chi||_1 ||v-\P v||_1 \\
&=& \tilde{C} h ||\psi||_2 ||v-\P v||_1 \\
&\leq& C h^s ||\psi||_2 ||v||_s. 
\end{eqnarray*}
Now using the decomposition \eqref{decompo} of $\I$, we get 
\begin{equation*}
||\I[\psi]] \leq \frac{3}{2}r^{2/3} ||\psi''|| + C(r) ||\psi'||.  
\end{equation*}
Taking  $r$ sufficiently small and using the coercivity, we obtain the regularity estimate \\$||\psi||_2 \leq C ||v-\mathcal{P}|| $
which yields 
\begin{equation*}
||v-\P v|| \leq C h^s ||v||_s, \forall v \in H^{s}_{per}. 
\end{equation*}
This completes the proof of this Lemma. \\

\end{proof}

\section{Discretization with respect to the space variable \label{spdisc}}
Motivated by \eqref{formva} we define the semidiscrete approximation $u_h(t,\cdot) \in \mathcal{S}^r_{h}$, $t\in (0,T)$, to $u$ by 
\begin{equation}
\left\{
    \begin{aligned}
        &  (u_{ht},v_h) + (u_h u_{hx},v'_{h}) + (u_{hx},v'_{h}) - (\J[u_h],v'_{h}) = 0,\quad  \forall v_h \in \mathcal{S}^r_h, t \in (0,T) \\
	& u_h(0,x) = u_{h}^0(x), 
    \end{aligned}
\right.
\label{semipb}
\end{equation} 
where $u_{h}^0 \in \mathcal{S}^r_h$ is an approximation of $u_0$ and $u^0_h$ is such that 
\begin{equation}
||u_h^0-u_0|| \leq C h^{r-1}. 
\label{CI}
\end{equation}
The semidiscrete approximation has the following property 
\begin{equation}
||u_h(t,\cdot)|| \leq e^{w_0t} ||u_h^0||, \quad t \in (0,T).
\label{l2estimatediscret}
\end{equation}
This inequality can be proved in the same way as Proposition \ref{prop1}. Now since $\mathcal{S}^r_h$ is finite-dimensional  we have 
\begin{equation}
\max_{t\in(0,T)} ||u_h(t,\cdot)||_\infty \leq C(h). 
\end{equation}
Then, regarding the equation \eqref{semipb} as a system of ODE, we deduce existence and uniqueness of the semidiscrete approximation $u_h$.

\begin{theorem} Let the solution $u$ of   \eqref{fowlereqn2} sufficiently smooth, and let \eqref{CI} hold. Then
\begin{equation}
\max_{t\in[0,T]} ||u(t,\cdot)-u_h(t,\cdot)|| \leq C h^{r-1},
\end{equation}  
where $C=C(u)$ is a positive constant. 
\end{theorem}

\begin{proof}
Let 
$$u-u_h=u-\P u + \P u - u_h = \rho + \mathcal{V},$$
where $\mathcal{P}$ is the operator projection defined in \eqref{operproj}. 
By Lemma \ref{lemproj}, we have 
\begin{equation*}
\max_{t \in [0,T]} ||\rho(t,\cdot)|| \leq C h^{r} 
\end{equation*}
Thus, it remains to estimate $||\v(t,\cdot)||$. 
\begin{eqnarray*}
(\v_t,\chi) + a(\v,\chi) &=& (\P u_t - (u_h)_t,\chi) + a(\P u,\chi) - a(u_h,\chi) \\
&=& (\P u_t,\chi) + a(\P u,\chi) - ((u_h)_t, \chi) - a(u_h,\chi) \\
\end{eqnarray*}
but since, $\forall \chi \in \mathcal{S}^r_h$
$$a(\P u,\chi) = a(u,\chi ).  $$
then
\begin{eqnarray*}
(\v_t,\chi) + a(\v, \chi) &=& (\P u_t,\chi) + a(u,\chi) + (u_h (u_h)_x,\chi) - \lambda (u_h,\chi) \\
&=& -(\rho_t,\chi) - (uu_x-u_h (u_h)_x,\chi) + \lambda (u-u_h,\chi) \\
&=&  -(\rho_t,\chi)  + \lambda (\rho,\chi) + \lambda (\v,\chi) - (u u_x - u_h (u_h)_x, \chi),
\end{eqnarray*}
i.e. 
\begin{eqnarray*}
(\v_t,\chi) + (\v_x,\chi') - (\J[\v],\chi') &=& -(\rho_t,\chi) + \lambda (\rho,\chi) - (u u_x- u_h (u_h)_x,\chi). 
\end{eqnarray*}
Taking $\chi = \v$, we obtain 
\begin{eqnarray*}
\frac{1}{2} \frac{d}{dt}||\v(t,\cdot)||^2 + ||\v_x(t,\cdot)||^2 - (\J[\v],\v_x) &=& -(\rho_t,\v) + \lambda (\rho,\v) -  (u u_x- u_h (u_h)_x,\v) \\
&=& -(\rho_t,\v) + \lambda (\rho,\v) -   (u(u-u_h)_x,\v) - (u_{hx}(u-u_h),\v )   
\end{eqnarray*}
Therefore, we have
\begin{eqnarray*}
\frac{1}{2} \frac{d}{dt} ||\v(t,\cdot)||^2 + ( \frac{1}{2} \v_x - \J [\v], \v_x )  + \frac{1}{2}||\v_x(t,\cdot)||^2 &\leq&
||\rho_t|| \, ||\v|| + \lambda ||\rho|| \, ||\v||  \\ 
&\,& + C \left\{ ||\rho|| + ||\rho_x|| + ||\v|| + ||\v_x||  \right\}  ||\v||, \\
&\leq &  \frac{1}{2} ||\v_x||^2 + \tilde{C}  \left( ||\rho||^2 + ||\rho_x||^2 + ||\rho_t||^2 +  ||\v||^2    \right). 
\end{eqnarray*}
Since  $||\rho|| \leq C h^{r}$,  $||\rho_t|| \leq C h^{r}$ and  $||\rho_x|| \leq C h^{r-1}$ (see Lemma \ref{lemproj})  we have 
\begin{eqnarray*}
\frac{1}{2} \frac{d}{dt} ||\v(t,\cdot)||^2 - w_0 ||\v(t,\cdot)||^2  \leq \tilde{C} h^{2(r-1)} + C' ||\v(t,\cdot)||^2 . 
\end{eqnarray*}
Therefore, we obtain 
\begin{eqnarray*}
\frac{1}{2} \frac{d}{dt} ||\v(t,\cdot)||^2  \leq \C h^{2(r-1)} + C ||\v(t,\cdot)||^2, 
\end{eqnarray*}
and Gronwall's lemma yields 
\begin{equation*}
\max_{t \in [0,T]} ||\v(t,\cdot)|| \leq c h^{r-1}, 
\end{equation*}
which concludes the proof of this theorem. 

\end{proof}

\section{Crank-Nicolson discretization \label{tpdisc}}

We investigate the following  second-order in time fully discrete finite element method for \eqref{fowlereqn2}. \\

Let $N \in \N, \Delta t:=\frac{T}{N}$ and $t^n := n\Delta t, n=0,\cdots, N.$ \\
For $u(t,\cdot) \in L^2(-L,L)$ and $  t \in [0, T],$ let
$$U^n := u(t^n,\cdot), \quad \partial U^n = \frac{U^{n+1}-U^{n}}{\Delta t}, \quad \mbox{ and } \quad U^{n+1/2} :=\frac{U^n+U^{n+1}}{2}.$$
The Crank-Nicolson approximations $U^n \in S^r_h$ to $u(t^n,\cdot)$ are given by $\forall n=0,\cdots,N-1,$
\begin{equation}
\left\{
    \begin{aligned}
 &  (\partial U^n, \chi ) + (U^{n+1/2} U_x^{n+1/2},\chi ) + (U_x^{n+1/2}, \chi' ) - (\J[U^{n+1/2}],\chi') = 0, \quad \forall \chi \in \mathcal{S}^r_h   \\
      & U^0 := u^0_h
    \end{aligned}
\right.
\label{schemacrank}
\end{equation}

In this section, we prove  the existence of the Crank-Nicolson approximations $U^1,\cdots,U^N$, derive the error estimate and show uniqueness of the Crank-Nicolson approximations. We also give a stability result for this scheme. 

The proof of the existence of the Crank-Nicolson approximations \eqref{schemacrank} is based on the following variant of the Brouwer fixed-point theorem: 

\begin{lemma}[ Browder, \cite{Br65}] 
Let $(H,(\cdot,\cdot)_H)$ be a finite-dimensional inner product space and denote by $||\cdot||_H$ the induced norm. Suppose that $g:H\rightarrow H$ is continuous and there exists an $\alpha >0$ such that $(g(x),x)_H>0 $ for all $x\in H$ with $||x||_H = \alpha.$ Then there exists $x^* \in H$ such that $g(x^*)=0$ and $||x^*||\leq \alpha.$ \\
\label{brouwer}
\end{lemma}

\begin{proposition}[Existence]
For  $\Delta t>0$  sufficiently small, there exists a  solution $U^n \in \mathcal{S}^r_h$ satisfying \eqref{schemacrank}. 

\end{proposition}

\begin{proof}
We prove the existence of $U^0,\cdots,U^N$ by induction. \\
Assume that $U^0,\cdots,U^n,$ for $n<N$ exist and let $g:\mathcal{S}^r_h \rightarrow \mathcal{S}^r_h$ be defined by 
\begin{equation*}
(g(V),\chi) = 2 (V-U^n,\chi) + \Delta t (V V',\chi) + \Delta t (V',\chi') - \Delta t (\J[V],\chi'), \quad \forall V,\chi \in \mathcal{S}^r_h. 
\end{equation*}
We can easily see that this mapping is continuous. Moreover, taking $\chi=V$ we have
\begin{equation*}
(g(V),V) = 2(V-U^n,V) + \Delta t ||V'||^2 - \Delta t (\J[V],V'),
\end{equation*}
and using Remark \ref{rem1} (which is still valable in $\mathcal{S}^r_h$), we obtain
\begin{equation*}
(g(V),V) \geq 2 ||V|| \left\{ (1-\frac{\alpha_0 \Delta t}{2})||V|| - ||U^n|| \right\}, \quad \forall V \in \mathcal{S}^r_h. 
\end{equation*}
Therefore, assuming $\Delta t < \frac{\alpha_0}{2}$ and for $V = \frac{2}{2-\alpha_0 \Delta t } U^n + 1$, we obtain $(g(V),V) >  0$. The existence of a $V^* \in \mathcal{S}_h$ such that $g(V^*)=0$ follows from Lemma \ref{brouwer}. Finally, $U^{n+1}:=2V^* - U^n $ satisfies \eqref{schemacrank}. 
\end{proof}

Uniqueness is less obvious, we need first to show an error estimate to get it. We will show it after the main theorem. \\

The time discretization being semi-implicit, we need a stability condition to ensure the validity of the computations. We then prove that the numerical process \eqref{schemacrank} is stable in the following sense: 
\begin{definition}[C-stability]
A numerical scheme is C-stable for the norm $||\cdot||$ if for all $T>0$, there exists a constant $K(T)>0$ independent of the time and space steps $\Delta t, h$ such that for all initial data $U^0$ 
\begin{equation}
||U^n|| \leq K(T) \, ||U^0||, \quad \forall 0 \leq n \leq \frac{T}{\Delta t}. 
\end{equation}
\end{definition}

\begin{proposition}[Stability \label{stabilitycondition}] Under the appropriate regularity assumptions, it exists two positive constants $C_1,C_2$ independent of $\Delta t, h$, and dependent of initial data, such that, if 
\begin{equation}
C_1 \frac{\Delta t}{ h^2} + C_2 \frac{\Delta t}{h^{4/3}} \leq 1,
\end{equation}
then the numerical scheme is C-stable.  \\
\end{proposition}

\begin{proof}
Taking $\chi = U^{n+1} $ in \eqref{schemacrank}, we obtain 
\begin{eqnarray}
(\frac{U^{n+1}-U^n}{\Delta t}, U^{n+1}) + (U^{n+1/2} U_x^{n+1/2},U^{n+1} ) + (U^{n+1/2}_x,U^{n+1}_x ) - (\J[U^{n+1/2}], U_x^{n+1} )  &=& 0 \nonumber \\
\label{eqpdn}
\end{eqnarray}
But
\begin{equation}
(U^{n+1}-U^n, U^{n+1}) = \frac{1}{2} ||U^{n+1}||^2 - \frac{1}{2}||U^n||^2 + \frac{1}{2}||U^{n+1}-U^n ||^2,
\label{est}
\end{equation}
and 
\begin{eqnarray*}
-(U^{n+1/2}_x,U^{n+1}_x ) + (\J[U^{n+1/2}], U_x^{n+1} ) & = & \frac{1}{2} (U_x^{n+1}-U_x^n,U_x^{n+1} ) - \frac{1}{2}(\J[U^{n+1}]-\J[U^n],U_x^{n+1} ) \\ 
&\, &  - (U_x^{n+1},U_x^{n+1}) + (\J[U^{n+1}],U_x^{n+1} ) \\
 & \leq &\frac{1}{4}||U_x^{n+1}-U_x^n||^2 + \frac{1}{4} ||\J[U^{n+1}]-\J[U^n]||^2  \\
&\,&  - (\frac{1}{4} U_x^{n+1} - \J[U^{n+1}], U^{n+1}_x ) - \frac{1}{4} ||U_x^{n+1}||^2 \\
&\leq &  \frac{1}{4}||U_x^{n+1}-U_x^n||^2 + \frac{1}{4} ||\J[U^{n+1}]-\J[U^n]||^2 \\
&\,&  + \alpha_0 ||U^{n+1}||^2 - \frac{1}{4} ||U_x^{n+1}||^2 ,
\end{eqnarray*}
where $\alpha_0 >0. $ From Lemma \ref{estimateJ1} and from inverse inequatity, we have 
\begin{equation}
||(u_h)_x||^2 \leq \frac{C}{h^2} ||u_h||^2, \quad ||\J[u_h]||^2 \leq \frac{C}{h^{4/3}} ||u_h||^2 \hspace{0.4cm} \forall u_h \in \mathcal{S}^r_h,
\label{relationx}
\end{equation}
then
\begin{eqnarray*}
-(U^{n+1/2}_x,U^{n+1}_x ) + (\J[U^{n+1/2}], U_x^{n+1} )  &\leq&  \frac{C_1}{h^2} ||U^{n+1}-U^n||^2 + \frac{C_2}{h^{4/3}} ||U^{n+1}-U^n||^2 \\ 
&+& \alpha_0 ||U^{n+1}||^2  - \frac{1}{4} ||U_x^{n+1}||^2
\label{estl}
\end{eqnarray*}
Let study now the nonlinear term. 
\begin{eqnarray*}
- 4 (U^{n+1/2} U^{n+1/2}_x, U^{n+1} ) &=& ((U^{n+1}-U^{n})(U_x^{n+1}-U_x^n),U^{n+1} ) - 2 (U^nU_x^n, U^{n+1} ) \\
&=&  ((U^{n+1}-U^{n})(U_x^{n+1}-U_x^n),U^{n+1} ) + (U^nU_x^{n+1}, U^{n} ) \\
&=&  ((U^{n+1}-U^{n})(U_x^{n+1}-U_x^n),U^{n+1} ) + (U^n U_x^{n+1},U^n-U^{n+1} ) \\ 
& & + ( (U^n-U^{n+1})U^{n+1}_x,U^{n+1} ), 
\end{eqnarray*}
by the boundedness of $U^{n+1}$ and $U^n$, we obtain
\begin{eqnarray}
\label{estnl}
-(U^{n+1/2} U^{n+1/2}_x, U^{n+1} ) &\leq& C ||U^{n+1}-U^{n}||  \, ||U_x^{n+1}-U_x^{n}|| \nonumber \\ 
&\, & + \tilde{C} ||U_x^{n+1}|| \, ||U^n-U^{n+1}||. 
\end{eqnarray}
Therefore, using \eqref{eqpdn}, \eqref{est}, \eqref{relationx} and \eqref{estl}, we get
\begin{eqnarray*}
(1-2\alpha_0 \Delta t)||U^{n+1}||^2 - ||U^n||^2 + (1 - \frac{C_1 \Delta t}{h^2} - \frac{C_2 \Delta t}{h^{4/3} } ) ||U^{n+1}-U^n ||^2 \leq C_3 \Delta t  ||U^{n+1}-U^n ||^2. 
\end{eqnarray*} 
Under the condition 
\begin{equation*}
1 - \frac{C_1 \Delta t}{h^2} - \frac{C_2 \Delta t}{h^{4/3} }  \geq 0, 
\end{equation*}
namely 
\begin{equation*}
\frac{C_1 \Delta t}{h^2} + \frac{C_2 \Delta t}{h^{4/3} } \leq 1,
\end{equation*}
we have 
\begin{equation*}
||U^{n+1}||^2 \leq (1+ C \Delta t) ||U^n||^2 \leq e^{C T} ||U^0||^2,
\end{equation*}
which shows that the numerical scheme is C-stable. \\
\end{proof}

\noindent The main result of this papier is given in the following theorem:

\begin{theorem}[Error estimate \label{maintheorem}] Let the solution $u$ of \eqref{fowlereqn2} be sufficiently smooth, $U^0,\cdots,U^N$ satisfy 
\eqref{schemacrank} and \eqref{CI} hold.  Then, for $\Delta t$ sufficiently small, we have 
\begin{equation}
\max_{0\leq n \leq N} ||u^n - U^n|| \leq C( \Delta t^2 + h^{r-1}), 
\end{equation}
where $C=C(u)$ is a positive constant. 
\end{theorem}

\begin{proof}
Let $W^n := \P u(t^n,\cdot)$, $\rho^n := u^n-W^n$ and $\v^n :=W^n-U^n$. Then
$$u^n-U^n = \rho^n + \v^n.$$
Using Lemma \ref{lemproj}, we have
$$\max_{0\leq n\leq N} ||\rho^n|| \leq C h^{r}.$$
Let us now estimate $||\v^n||.$
\begin{eqnarray*}
(\partial \v^n,\chi) + a(\v^{n+1/2},\chi) &=&  (\partial W^n, \chi) + a(W^{n+1/2},\chi ) - (\partial U^n,\chi) - a(U^{n+1/2},\chi) \\
\end{eqnarray*}
and since
$  a(W^{n+1/2},\chi ) = a(u^{n+1/2}, \chi)$ and   $$  (\partial U^n,\chi) + a(U^{n+1/2},\chi) = - (U^{n+1/2} U^{n+1/2}_x,\chi )+ \lambda (U^{n+1/2},\chi) $$ then
\begin{eqnarray}
\label{FVeq}
&\,& (\partial \v^n,\chi) + a(\v^{n+1/2},\chi) =  (\partial W^n, \chi) + a(u^{n+1/2}, \chi) + (U^{n+1/2} U^{n+1/2}_x,\chi ) - \lambda (U^{n+1/2},\chi)  \nonumber \\
& =& (\partial W^n, \chi)  - (u_t^{n+1/2},\chi) - (u^{n+1/2} u^{n+1/2}_x,\chi )  + \lambda (u^{n+1/2},\chi) + (U^{n+1/2} U^{n+1/2}_x,\chi ) - \lambda (U^{n+1/2},\chi)  \nonumber \\
&=& (w_1 + w_2 + w_3, \chi) + \lambda (\rho^{n+1/2},\chi) + \lambda (\v^{n+1/2},\chi) \nonumber\\
\end{eqnarray}
with 
$w_1 := \partial W^n - \partial u^n$, $w_2 :=\partial u^n - u_t^{n+1/2}$ and $w_3 := U^{n+1/2} U^{n+1/2}_x-u^{n+1/2} u^{n+1/2}_x.$
We have that $||w_1|| \leq C h^{r}.$ \\
\noindent Let us study $w_2.$ We have 
\begin{eqnarray*}
\Delta t \,  w_2 &=& u^{n+1} - u^n -\Delta t u_t^{n+1/2}\\
&=& \frac{1}{2} \int_{t^n}^{t^{n+1/2}} (s-t_n)^2 u_{3t}(s) ds  + \frac{1}{2} \int_{t^{n+1/2}}^{t^{n+1}}  (s-t_{n+1})^2 u_{3t}(s) ds \\
&\leq& C \Delta t^2  \int_{t^n}^{t^{n+1}} ||u_{3t}(s)|| ds. 
\end{eqnarray*}

\noindent Let us study $w_3$: Since
\begin{eqnarray*}
w_3 &=&  U^{n+1/2} U^{n+1/2}_x-u^{n+1/2} u^{n+1/2}_x \\
&=& U^{n+1/2}(U_x^{n+1/2} - u_x^{n+1/2} ) + u_x^{n+1/2}(U^{n+1/2}-u^{n+1/2}) \\
\end{eqnarray*}
then 
\begin{equation*}
||w_3|| \leq  ||U^{n+1/2}||_{\infty} ||U_x^{n+1/2} - u_x^{n+1/2}|| + ||u_x^{n+1/2}||_{\infty} ||U^{n+1/2}-u^{n+1/2}||
\end{equation*}
But,  $ U_x^{n+1/2} - u_x^{n+1/2} = \rho_x^{n+1/2} + \v_x^{n+1/2}$. \\

Now, taking $\chi = \v^{n+1/2}$ in \eqref{FVeq}, we get
\begin{eqnarray*}
&\,& (\partial \v^n,\v^{n+1/2}) + (\v_x^{n+1/2}, \v_x^{n+1/2}) - (\J[\v^{n+1/2}],\v^{n+1/2}) + \lambda (\v^{n+1/2},\v^{n+1/2}) = \\ 
&\,& (w_1,\v^{n+1/2})  +  (w_2,\v^{n+1/2}) + (w_3,\v^{n+1/2}) +  \lambda (\rho^{n+1/2},\v^{n+1/2}) + \lambda (\v^{n+1/2},\v^{n+1/2})
\end{eqnarray*}
and since 
$$(\partial \v^n,\v^{n+1/2}) = \frac{1}{2 \Delta t} ||\v^{n+1}||^2 - \frac{1}{2 \Delta t} ||\v^{n}||^2,$$
then we have
\begin{eqnarray*} 
&\,& ||\v^{n+1}||^2 -||\v^{n}||^2 + 2 \Delta t \left(  ||\v_x^{n+1/2}||^2 - (\J[\v^{n+1/2}], \v_x^{n+1/2}) +  \lambda ||\v^{n+1/2}||^2  \right)   \leq  2 \Delta t ||w_1|| ||\v^{n+1/2}|| \\
&\, & +  2 \Delta t \left(   ||w_2|| \,  ||\v^{n+1/2}|| + ||w_3|| \,  ||\v^{n+1/2}||   \right)   + 2\Delta t  \, \lambda ||\rho^{n+1/2}|| \, ||\v^{n+1/2}|| + 2 \Delta t \lambda ||\v^{n+1/2}||^2.  \\
\end{eqnarray*}
Using Lemma \ref{lemproj} and Remark \ref{rem1} we have
\begin{eqnarray*}
||\v^{n+1}||^2-||\v^{n}||^2  &\leq&  \Delta t C(u)(h^{2(r-1)} + \Delta t^4  +  ||\v^{n+1/2}||^2 )
\end{eqnarray*}
Since $4 ||\v^{n+1/2}||^2 = ||\v^{n+1}||^2 + ||\v^n||^{2} + 2(\v^{n+1},\v^n)$ then 
for $\Delta t$ sufficiently small and using the discrete Gronwall lemma, we get
$$\max_{0 \leq n \leq N} ||\v^n|| \leq c(u) (\Delta t^2 + h^{r-1}),$$
which concludes the proof. 
\end{proof}

\begin{remark}[Uniqueness] We return to the question of uniqueness of the solution of \eqref{schemacrank}. We show that this holds for  $\Delta t, h$  sufficently small  when the solution of the continuous problem is smooth and when \eqref{CI} holds. \\
Let $U^n$ and $V^n$ be two solutions of \eqref{schemacrank} with $U^{n-1}$ given. Letting $E^n := U^n-V^n$ , we obtain by subtraction 
\begin{eqnarray*}
(\partial E^n,\chi) + (E_x^{n+1/2},X') - (\J[E^{n+1/2}],X') =  (E^{n+1/2} E_x^{n+1/2},\chi) + (U^{n+1/2} E^{n+1/2},\chi') \hspace{0.1cm}    \, \forall \chi \in \mathcal{S}^r_h.
\end{eqnarray*}
Taking $\chi = E^{n+1/2}$ we obtain by periodicity 
\begin{eqnarray*}
\frac{1}{2 \Delta t} (||E^{n+1}||^2-||E^n||^2) &+&   ||E_x^{n+1/2} ||^2 - ( \J [E^{n+1/2}], E^{n+1/2} ) = \\
 &=& (U^{n+1/2} E_x^{n+1/2},E_x^{n+1/2}) \\
&\leq& \frac{1}{2} ||U^{n+1/2}||_{\infty}^2 ||E^{n+1/2}||^2 + \frac{1}{2} ||E_x^{n+1/2} ||^2 \\
&\leq& \frac{1}{2} (||W^{n+1/2}||^2_{\infty}  + || \mathcal{V}^{n+1/2}||_{\infty}^2) ||E^{n+1/2}||^2  + \frac{1}{2} ||E_x^{n+1/2} ||^2.
\end{eqnarray*}
Using Remark \ref{rem1}, Theorem \ref{maintheorem} and since the following inverse inequality holds 
\begin{equation}
||\chi||_{\infty} \leq C h^{-1/2} ||\chi||, \quad \forall \chi \in \mathcal{S}^r_h, 
\end{equation}
we obtain 
\begin{eqnarray*}
\frac{1}{2 \Delta t} (||E^{n+1}||^2-||E^n||^2) &\leq&  C  ( 1 + \Delta t^4 + h^{2(r-2)}) ||E^{n+1/2} ||^2 \\
&\leq&  C  ( 1 + h^{-1/2}\Delta t^4 + h^{-1/2} h^{2(r-1)} ) (||E^{n} ||^2+  ||E^{n+1}||^2)
\end{eqnarray*} 
Therefore if we assume  $E^n=0$, we get  for  $\Delta t^5 h^{-1/2}$ and $\Delta t \, h^{2r-3/2} $ sufficiently small  $E^{n+1}=0.$ We deduce uniqueness of the Crank-Nicolson approximations. 

\end{remark}

\section{Numerical experiments \label{senum}}

We conclude this paper by presenting some experimental results obtained using numerical scheme \eqref{schemacrank} with Crank-Nicolson method for the time disretization and the Garlerkin method for different polynomial orders. In our numerical experiments we have imposed a zero Dirichlet boundary condition on the whole exterior domain $\left\{ |x| >1 \right\}$ and we have confined the nonlocal operator $\J$ to the domain $\Omega=\left\{ |x| \leq 1 \right\}$. This means we have computed the value of $U^{n+1}$ by using only the values $U^{n}(x_i)$ with $x_i \in \Omega$. \\ For all the numerical tests, the stability condition stated in Proposition \ref{stabilitycondition} is satisfied. \\
 In order to magnify the effect of the nonlocal term, we add a  small viscous coefficient $\varepsilon$ in the Fowler equation 
\begin{equation}
\partial_t u(t,x) +  \partial_x \left(\frac{u^2}{2} + \J[u] \right) (t,x) - \varepsilon \partial_{xx} u = 0,\\
\end{equation}
We consider the following two initial data:

\noindent Example 1: 
\begin{equation*}
u_0(x) = \left\{
\begin{array}{rl}
  0 &\mbox{ if } x\leq -0.6  \\
  4x + 2.4 & \mbox{ if } -0.6<x\leq -0.4 \\
  0.8 &\mbox{ if } -0.4 <x\leq 0. \\ 
  0.8-4x &\mbox{ if} 0. < x \leq 0.2 \\ 
  0. &\mbox{ if } x >0.2 \\
\end{array}
\right.
\end{equation*}

\noindent  Example 2: 
$$u_0(x) =  e^{-50(x+0.2)^2}.$$

The numerical results are presented in Figures \ref{fig1} and \ref{fig2}.  For the first example, we used linear elements for the Galerkin method ($r=2$) while we used a second order polynomial approximations $(r=3)$ for the second example.    In all plots, the solid line represents the initial datum while the dotted line the numerical solution at $t=T.$ As we expect from the viscous Burgers equation the initial data are propagated downstream but we observe here in addition an "erosive process" behind the bump due to the nonlocal term. \\

\begin{figure}[!t]
\centering
\includegraphics[scale=0.35]{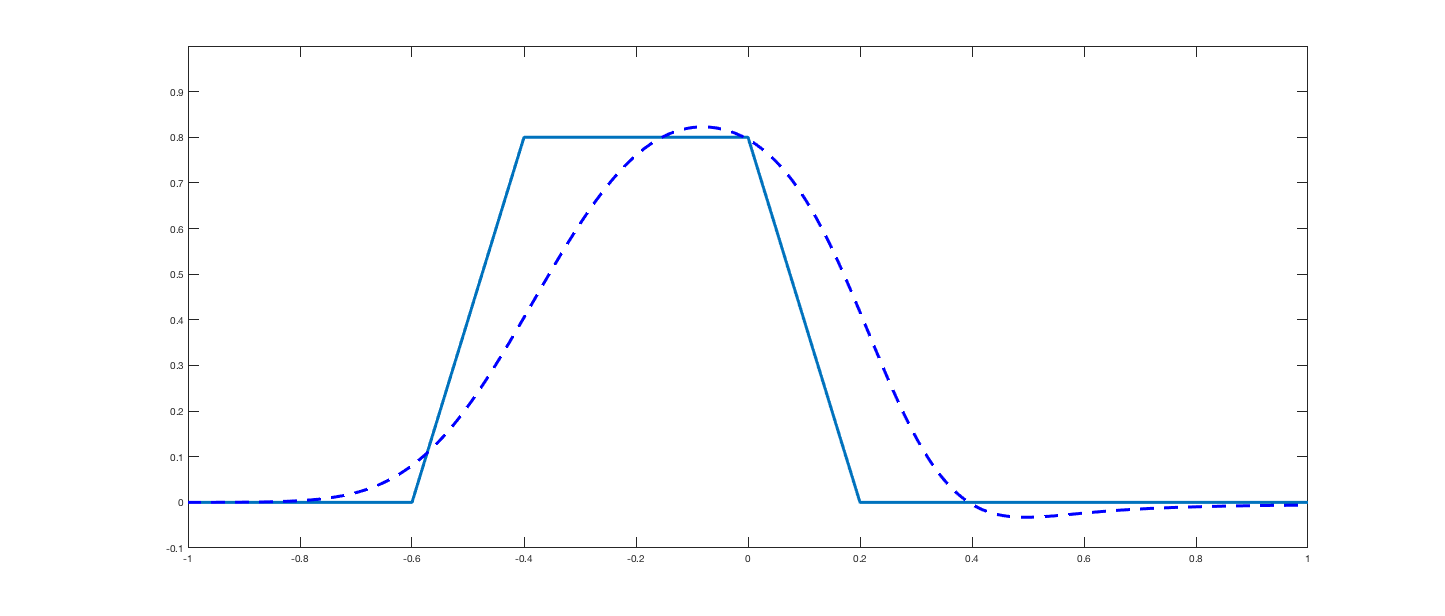}
\caption{Example 1: $r=2$, $T=0.1$ and $N=640$}
\label{fig1}
\end{figure}

\begin{figure}[!t]
\includegraphics[scale=0.35]{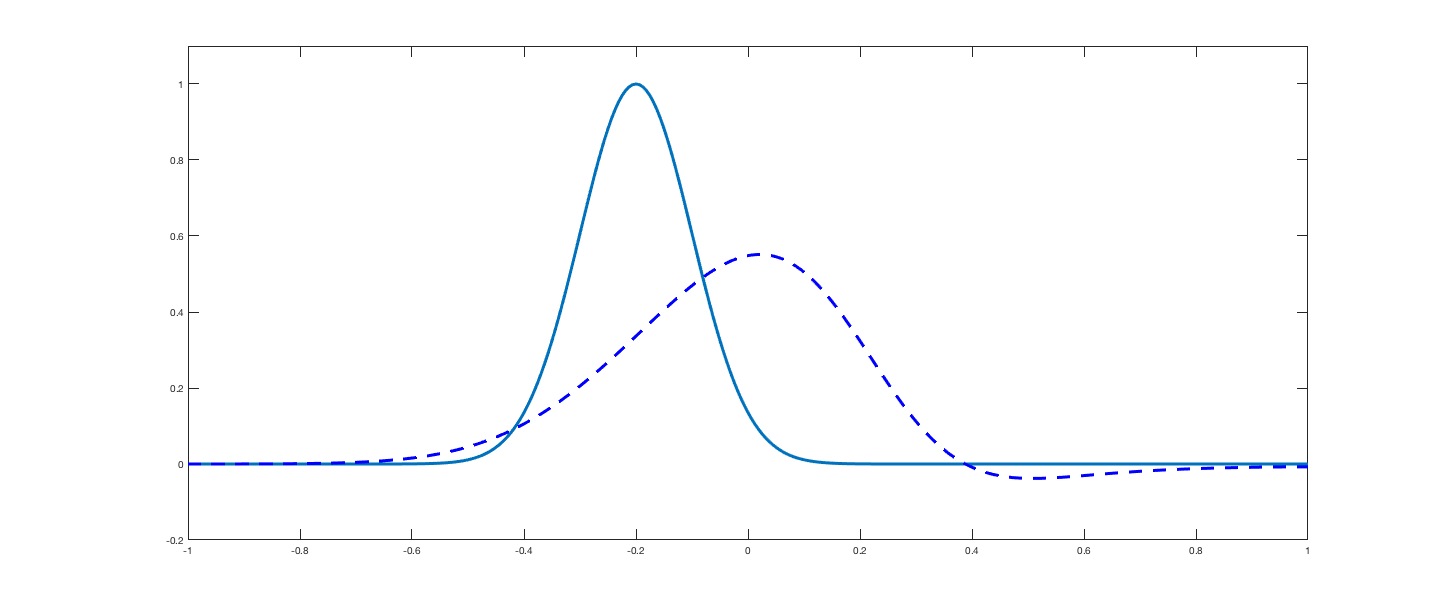}
\caption{Example 2: $r=3$, $T=0.2$ and $N=640$}
\label{fig2}
\end{figure}

\noindent The numerical rate of convergence for the solutions in Figures \ref{fig1} and \ref{fig2} are presented in Tables \ref{table1} and \ref{table2}. 

\begin{table}[h!]
\begin{tabular}{|c|c|c|c|}
  
  $N$ & error  & relative error & order \\
  \hline
  20 & 4.0759e-04 & 4.3296e-04 & 1.9532 \\
  40 & 1.0526e-04 & 1.1181e-04 & 1.9173 \\
  80 & 2.7867e-05 & 2.9601e-05 & 1.7207 \\
  160 & 8.4546e-06 & 8.9808e-06 & - \\
\end{tabular}
\vspace{0.4cm}
\caption{Example 1: Error, relative error and numerical rate of convergence  for one order polynomial approximations ($r=2$). $N$ denotes the number of elements.  }
\label{table1}
\end{table}

\begin{table}
\begin{tabular}{|c|c|c|c|}
  
  $N$ & error  & relative error & order \\
  \hline
  20 & 1.0381e-04 & 3.1458e-04 & 2.3097 \\
  40 & 2.0939e-05 & 6.3452e-05 & 2.0792 \\
  80 & 4.9551e-06 & 1.5015e-05 & 1.8057 \\
  160 & 1.4174e-06 & 4.2952e-06 & - \\
\end{tabular}
\vspace{0.4cm}
\caption{Example 2: Error, relative error and numerical rate of convergence for second order polynomial approximations ($r=3$). $N$ denotes the number of elements.  }
\label{table2}
\end{table}
\noindent We have measured the $L^2$-error 
\begin{equation*}
E_h = ||u_h(T,\cdot)-\hat{u}_e(T,\cdot)||^2,
\end{equation*}
where $\hat{u}_e$ is the numerical solution which has been computed using a very fine grid $h=2/640$. 
We also have measured the relative error 
\begin{equation*}
R_h = \left( \frac{1}{||\hat{u}_e(T,\cdot)||^2}  \right) E_h, 
\end{equation*}
and the approximation rate of convergence
\begin{equation*}
\alpha_h = \left( \frac{1}{\log 2}  \right)  \left( \log E_h - \log E_{h/2} \right). 
\end{equation*}
We observe that the order of convergence is reached, confirming the theoretical results. Indeed, the experimental rates of convergence are greather than one for the first numerical example ($r=2$) and for the second example ($r=3$),  the numerical rates of convergence  are near to 2.

\bibliographystyle{siam}
\bibliography{biblio}

\end{document}